\documentclass{article} 
\usepackage{amsmath,amsthm,amssymb,enumerate,comment,color}

\theoremstyle{plain}
\newtheorem{thm}{Theorem}[section]
\newtheorem{lem}[thm]{Lemma}
\newtheorem{prop}[thm]{Proposition}
\newtheorem{cor}[thm]{Corollary}

\theoremstyle{definition}
\newtheorem{defi}[thm]{Definition}

\theoremstyle{remark}
\newtheorem{rem}[thm]{Remark}
\newtheorem{eg}[thm]{Example}

\newcommand{\R}{\mathbb{R}}
\newcommand{\Z}{\mathbb{Z}}

\newcommand{\Q}{\mathbb{Q}}
\newcommand{\C}{\mathbb{C}}

\newcommand{\Symp}{\mathop{\mathrm{Symp}}\nolimits}
\newcommand{\Spec}{\mathop{\mathrm{Spec}}\nolimits}
\newcommand{\Ham}{\mathop{\mathrm{Ham}}\nolimits}
\newcommand{\Crit}{\mathop{\mathrm{Crit}}\nolimits}
\newcommand{\CZ}{\mathop{\mathrm{CZ\mathchar`-ind}}\nolimits}
\newcommand{\dist}{\mathop{\mathrm{dist}}\nolimits}
\newcommand{\Sp}{\mathop{\mathrm{Sp}}\nolimits}
\newcommand{\Id}{\mathop{\mathrm{Id}}\nolimits}
\newcommand{\sign}{\mathop{\mathrm{sign}}\nolimits}
\newcommand{\Ker}{\mathop{\mathrm{Ker}}\nolimits}
\newcommand{\Hom}{\mathop{\mathrm{Hom}}\nolimits}
\newcommand{\Fix}{\mathop{\mathrm{Fix}}\nolimits}
\newcommand{\supp}{\mathop{\mathrm{supp}}\nolimits}
\newcommand{\signn}{\mathop{\widetilde{\mathrm{sign}}}\nolimits}
\newcommand{\rank}{\mathop{\mathrm{rank}}\nolimits}

\newcommand{\gl}{\mathop{\mathfrak{gl}}\nolimits}
\newcommand{\tocong}{\stackrel{\cong}{\to}}

\title{Spectral invariants of distance functions}
\author{Suguru Ishikawa}
\date{October 15, 2015}

\begin{document}

\maketitle

\begin{abstract}
Calculating the spectral invariant of Floer homology of the distance function, we can
find new superheavy subsets in symplectic manifolds.
We show if convex open subsets in Euclidian space with the standard symplectic form are
disjointly embedded in a spherically negative monotone closed symplectic manifold,
their compliment is superheavy.
In particular, the $S^1$ bouquet in a closed Riemann surface with genus $g\geq 1$ is superheavy.
We also prove some analogous properties of a monotone closed symplectic manifold.
These can be used to extend Seyfaddni's result
about lower bounds of Poisson bracket invariant.
\end{abstract}

\section{Introduction and main results}
In \cite{EP09}, Entov and Polterovich introduced the notion of a (super)heavy set of
a symplectic manifold, which enabled them to find a lot of examples of non-displaceable sets
in symplectic manifolds.
Heavy sets cannot be displaced by Hamiltonian isotopy, and
superheavy sets cannot be displaced by symplectic isotopy.
An important fact proved in \cite{EP09} is that (super)heaviness is
preserved by product, which produces a lot of examples of superhevy sets.

Superheavyness of a closed subset of a symplectic manifold is equivalent to
the vanishing of the partial symplectic quasi-state of the distance function
from that subset (see Section \ref{review}).
Hence we can detect superheavy subsets by the direct calculation of
the partial symplectic quasi-states of the special Hamiltonians.

We show in this paper the following theorem.
\begin{thm}\label{main}
Let $(M,\omega)$ be a closed symplectic manifold with dimension $2n$ and
assume $c_1(TM) = \kappa \omega$ on $\pi _2(M)$, $\kappa \leq 0$.
If convex open subsets $U_j \subset (\R^{2n},\omega_0=\sum_i dx_i\wedge dy_i)$ are
symplectically embedded in $M$, then $X=M\setminus \coprod_{j}U_j$ is superheavy
with respect to every non-zero idempotent of $QH(M,\omega)$.

\end{thm}
The proof of the above theorem is based on the estimate of the Conley-Zehnder index
of periodic orbits of the distance-like function.

In the above theorem, the assumption about $c_1$ is necessary.
Indeed, the complex projective space $(\C P^n,\tau_0)$ with the Fubini-Study form is an easy counterexample since $\C P^{n-1} = \C P^n\setminus B(1) \subset \C P^n$ is not superheavy (see Example \ref{counter}). However, we can show the analogous statement for a monotone closed symplectic manifold $(M, \omega)$ if $U_j$ are sufficiently small (but not necessarily displaceable by Hamiltonian isotopy).

The above theorem is a corollary of Proposition \ref{bound} proved in Section \ref{proof}.
Seyfaddini pointed out that our Proposition \ref{bound} is a generalization of Theorem 2 of his paper \cite{Sey14}. In particular, the dispalceability assumption in his theorem can be removed in the case of spherically negative monotone symplectic manifold.

Seyfaddini and Polterovich pointed out that Theorem 4.8 of \cite{Pol14} and
its extension Theorem 9 of \cite{Sey14} can
be farther extended by our Proposition \ref{bound}. (See section \ref{application}.)
The method used to prove Proposition \ref{bound} can be used to get some bound of the Hofer-Zehnder capacity of some open subset of a monotone closed symplectic manifold. (See section \ref{another}.)

\section{Review}
\label{review}
For a closed symplectic manifold $(M,\omega)$, let $\Omega_0M$ be the connected component
of the loop space of $M$, and
$\widetilde{\Omega_0M}=\{(x,u); x\in \Omega _0M, u:D\to M, u|_{\partial D}=x\}/\sim$
be its standard covering space,
where $(x,u)\sim (y,v)$ if and only if $x=y$, $c_1(u\# \bar{v})=0$ and 
$\omega(u\# \bar{v})=0$ ($c_1$ stands for $c_1(TM)$).
$\Gamma =\pi _2 (M)/\Ker \omega \cap \Ker c_1$ acts on $\widetilde{\Omega_0M}$ by
$[x,u]\cdot A = [x, (-A)\# u]$.
Every Hamiltonian $H:M\times S^1\to \R$ defines the action functional
$\mathcal{A}_H:\widetilde{\Omega_0M}\to\R$ by
\[
\mathcal{A}_H([x,u])=-\int_Du^\ast \omega + \int _0^1H(x(t),t)dt.
\]
The set of its critical points is $\Crit \mathcal{A}_H=\{[x,u]\in\widetilde{\Omega_0M};x(t)=\phi^H_tx(0)$
is a contractible periodic orbit$\}$,
where $\phi^H_t$ denotes the flow of the Hamiltonian vector field $X_H$ of $H$.
$X_H$ is defined by $i_{X_H}\omega = -dH$.
We note that $\mathcal{A}_H([x,u]\cdot A) =\mathcal{A}_H([x,u]) + \omega (A)$.

For each $[x,u]\in \Crit \mathcal{A}_H$, $\CZ _H([x,u]) \in \frac{1}{2}\Z$ denotes
the Conley-Zehnder index of the linearized flow $\{(\phi^H_t)_{\ast x(0)}:T_xM\to T_{\phi^H_t(x(0))}M\}_{t\in[0,1]}$
under the symplectic trivialization of $x^\ast TM$ given by $u$
(for Conley Zehender index, see \cite{CZ83} for nondegenerate orbits
and \cite{RS93} for general ones. See also Section \ref{CZ}.)
If it is nondegenerate, then its Conley-Zehnder index is an integer.
In this case we define its degree as $\mu ([x,u])=-\CZ([x,u])$.
We note that $\mu ([x,u]\cdot A) = \mu ([x,u]) + 2c_1(A)$.
For a nondegenerate Hamiltonian $H:M\times S^1\to \R$, we define
\begin{equation}
\Spec _kH=\{\mathcal{A}_H([x,u]); [x,u]\in \Crit \mathcal{A}_H, \mu([x,u])=k \},\label{spec_k}
\end{equation}
which is a countable subset of $\R$.

For a non-zero element $\alpha\in QH_{2n}(M,\omega)$ of its quantum homology group,
we have a spectral invariant
$c(\alpha;-):C(M\times S^1;\R)\to \R$, which satisfies the following properties:

\begin{itemize}
\item normalization:
\[
c(\alpha;0)=\nu(\alpha):=\inf_{\alpha=\sum_A \lambda_Ae^A}\max_{\lambda_A\neq 0} \{\omega (A)\}
\]

\item shift property:
$c(\alpha; F + a) = c(\alpha; F) + \int ^1_0 a(t)dt$ for $a : S^1 \to \R$.

\item monotonicity:
for $F,G\in C(M\times S^1;\R)$,
\[
\int_0^1\min_{M}(F_t-G_t)dt\leq c(\alpha;F)-c(\alpha;G) \leq \int_0^1\max_{M}(F_t-G_t)dt.
\]

The monotonicity property also implies Lipschitz continuity of the spectral invariant.

\item spectrality: for a nondegenerate Hamiltonian $H\in C(M\times S^1;\R)$, $c(\alpha;H) \in \Spec_{n} H$.

\item triangle inequality:
for $F,G\in C(M\times S^1;\R)$ and $\alpha, \beta \in QH_{2n}(M,\omega)$ such that $\alpha \beta \neq 0$,
\[
c(\alpha\beta;F\# G)\leq c(\alpha;F)+c(\beta;G),
\]
where $(F\# G)_t=F_t+G_t\circ (\phi ^F_t)^{-1}$.
In particular,
\[
c(\alpha;F) \leq c(\alpha; 0) + c([M]; F).
\]

\item symplectic invariance: for $\theta \in \Symp (M;\omega)$, $c(\alpha;\theta^\ast H)=c(\theta_\ast \alpha;H)$,
where $(\theta^\ast H)_t=H_t\circ \theta$.

\item homotopy invariance: if $F,G\in C(M\times S^1;\R)$ satisfies
\begin{align*}
\phi ^F_1=\phi ^G_1 \text{ in } \widetilde{\Ham}(M,\omega), \quad \langle F \rangle := \int ^1_0 dt \int_M H_t\omega^{\wedge n}= \langle G \rangle
\end{align*}
then $c(\alpha;F)=c(\alpha;G)$.

\end{itemize}
for details, see \cite{Sch00} for symplectically aspherical manifolds and \cite{Oh05} for general ones.

For a non-zero idempotent $\alpha \in QH_{2n}(M,\omega)$,
Entov and Polterovich \cite{EP03} defined 
the homogenization of the spectral invariant $\bar{c}(\alpha;-): C(M\times S^1;\R)\to \R$ by
\[
\bar{c}(\alpha;H)=\lim_{k\to \infty}\frac{1}{k}c(\alpha;kH_{kt})
\]
and they defined in \cite{EP06} partial symplectic quasi-state
$\zeta(\alpha;-): C(M;\R)\to \R$
\[
\zeta(\alpha;H)=\lim_{k\to \infty}\frac{1}{k}c(\alpha;kH)=\bar{c}(\alpha;H).
\]
They proved in \cite{EP07} the following lemma.
\begin{lem}[\cite{EP07}]
For $F,G\in C(M\times S^1;\R)$,
\[
\int_0^1\min_{M}(F_t-G_t)dt\leq \bar{c}(\alpha;F)-\bar{c}(\alpha;G) \leq \int_0^1\max_{M}(F_t-G_t)dt.
\]
In particular, for $F,G\in C(M;\R)$,
\[
\min_{M}(F-G)\leq \zeta(\alpha;F)-\zeta(\alpha;G) \leq \max_{M}(F-G).
\]
\end{lem} 

Other properties of the partial symplectic quasi-states which were proved in \cite{EP06} are
as follows: ($\zeta(-)$ is short for $\zeta(\alpha; -)$ below.)
\begin{itemize}
\item
$\zeta(a)=a$, $\zeta(H+a)=\zeta(H)+a$ for all $a\in\R$.

\item
$\zeta (aH)=a\zeta (H)$ for all $a\geq 0$.

\item
$\zeta (-H)+\zeta(H) \geq 0$.

\item
If $\supp H \subset M$ is stably displaceable then $\zeta(H)=0$.

\item
If $F,G\in C^1(M;\R)$ and $\{F,G\}=0$ then $\zeta(F+G)\leq \zeta(F)+\zeta(G)$.

\end{itemize}

\begin{defi}[\cite{EP09}]
A closed subset $X\subset M$ is $\alpha$-heavy if
\begin{align*} 
\zeta(\alpha;H)\geq \min_X H \text{ for all } H\in C(M;\R)
\end{align*}
and is $\alpha$-superheavy if
\begin{align*}
\zeta(\alpha;H)\leq \max_X H\text{ for all } H\in C(M;\R).
\end{align*}
\end{defi}
They proved the following important properties.
\begin{prop}[\cite{EP09}]\label{heavysuperheavy}
If $X \subset M $ is $\alpha$-superheavy and $Y \subset M$ is $\alpha$-heavy,
then $X \cap Y\neq \emptyset$.
\end{prop}
\begin{prop}[\cite{EP09}]
$\alpha$-heavy set $X\subset M$ is stably non-displaceable.
\end{prop}
\begin{prop}[\cite{EP09}]
$\alpha$-superheavy set is $\alpha$-heavy.
Hence $\alpha$-superheavy set cannot be displaced by
$\phi \in \Symp (M,\omega)$ such that $\phi^\ast \alpha =\alpha$.
\end{prop}

The following lemma is almost the same as that of \cite{EP09}.
\begin{lem}\label{dist}
For a closed subset $X\subset M$, the following conditions are equivalent:
\begin{enumerate}[\normalfont(i)]
\item $X$ is $\alpha$-heavy.
\item if $H\leq 0$ and $H|_X=0$, then $\zeta (\alpha ;H)=0$.
\item there exists some $H_0\in C(M;\R)$ such that $H_0\leq 0$, $\{H_0=0\}=X$ and $\zeta(\alpha;H_0)=0$.
$($$\{H_0=0\}$ is short for  $\{ x\in M ; H_0(x)=0\}$.$)$

\end{enumerate}
The following conditions are also equivalent:
\begin{enumerate}[\normalfont(i)$'$]
\item $X$ is $\alpha$-superheavy.
\item if $H\geq 0$ and $H|_X=0$, then $\zeta (\alpha ;H)=0$.
\item there exists some $H_0\in C(M;\R)$ such that $H_0\geq 0$, $\{H_0=0\}=X$ and $\zeta(\alpha;H_0)=0$.

\end{enumerate}
\end{lem}	
\begin{proof}
The equivalence of (i) and (ii) was proved in \cite{EP09}. (ii) trivially implies (iii).
If (iii) is satisfied, then for any $H\leq 0$ with $H|_X=0$ and $\epsilon >0$,
there exists some $a\geq0$ such that $H\geq aH_0-\epsilon$.
Hence $0\geq\zeta(\alpha;H)\geq a\zeta(\alpha;H_0)-\epsilon=-\epsilon$. 
The equivalence of (i)$'$ to (iii)$'$ is similar.
\end{proof}
\begin{rem}
In the above conditions (iii) and (iii)$'$, we can take $H_0(x)=\pm \dist(x,X)$.
Hence we can prove the superheavyness of a given closed subset
by the analysis of a single function. This trivial observation is crucial in this paper.
\end{rem}

\begin{prop}[\cite{EP09}]\label{product}
Let $\alpha \in QH_{2n}(M,\omega)$ and $\alpha'\in QH_{2n}(M',\omega')$ be non-zero idempotents,
and $H\in C(M\times S^1;\R)$ and $H' \in C(M'\times S^1;\R)$ be Hamiltonians.
Then we have
\begin{equation*}
c(\alpha \otimes \alpha';H+H')=c(\alpha;H)+c(\alpha';H').
\end{equation*}
Hence if $X\subset M$ is $\alpha$-(super)heavy and $X'\subset M'$ is $\alpha'$-(super)heavy,
then $X\times X'\subset M\times M'$ is also $\alpha \otimes \alpha'$-(super)heavy.
\end{prop}

We remark that the proof of the last claim of the above proposition in \cite{EP09} can be simplified
by using the fact that $\dist((x,x'),X\times X')=\dist(x,X)+\dist(x',X')$
for apropriate distance functions on $M$, $M'$ and $M\times M'$.
The following lemma is also related to this remark.
\begin{lem}\label{com}
Suppose $f_1, f_2, \cdots f_m:M\to \R$ are pairwise Poisson commuting functions
and $\{f_j=0\}\subset M$ $(j=1,2,\dots , m)$ are all $\alpha$-superheavy.
Then their intersection $\bigcap \{ f_j=0\}\subset M$ is also $\alpha$-superheavy.
\end{lem}
\begin{proof}
Since $\{f_j^2\}$ are also pairwise Poisson commuting,
\[
\zeta(\alpha;\sum_j f_j^2)\leq \sum_j \zeta(\alpha;f_j^2)=0.
\]
This implies the claim by Lemma \ref{dist}.
\end{proof}
The following lemma is also useful.
\begin{lem}\label{geq0}
If $F\in C(M)$ vanishes on some $\alpha$-heavy subset $X\subset M$,
then $c(\alpha;F)\geq 0$.
\end{lem}
\begin{proof}
By the triangle inequality of the spectral invariant,
\[
c(\alpha; F)\geq \zeta(\alpha ; F) \geq \min_X F \geq 0.
\]
\end{proof}

\section{Estimates of the Conley-Zehnder index}\label{CZ}
We prove some properties of the Conley-Zehnder index of a path in $\Sp(2n)$.
The Conley-Zehnder index was introduced in \cite{CZ83} for a nondegenerate path,
and Robbin and Salamon generalized it for every path in \cite{RS93}.
First we describe two equivalent definitions of the index along the same line
as in \cite{RS93} and \cite{RS95}(for details, see these articles).

The first definition is related to the spectral flow of the associated selfadjoint operator. 
First we note that for any path $\Phi(t)$ in $\Sp(2n)$ with $\Phi(0) = \Id$,
there exists a path of symmetric matrices $S(t) : [0,1] \to \gl(2n;\R)$
such that $\Phi$ is the fundamental solution of the equation
$\dot{x}=J_0S(t)x$, that is,
\[
\partial _t \Phi(t) = J_0S(t) \Phi(t), \ \Phi(0) = \Id.
\]
Let $\{S^s(t):[0,1]\to \gl(2n;\R)\}_{s\in[0,1]}$ be a family of paths of symmetric matrices
and $\{\Phi^s(t):[0,1]\to \Sp(2n)\}_{s\in[0,1]}$ be the fundamental solutions of the equations $\dot{x}=J_0S^s(t)x$.
Define a family of selfadjoint operators by
$A(s)=J_0\partial_t+S^s(t):W^{1,2}(S^1;\R^{2n})\to L^2(S^1;\R^{2n})$.
$s\in[0,1]$ is called a crossing if $A(s)$ has a nontrivial kernel. In that case
we define its crossing operator $\Gamma(s,A):N(A(s))\to N(A(s))$ by
\[
\Gamma(s,A)\xi=\pi_{N(A(s))}(\partial_sS^s\xi),
\]
where $\pi_{N(A(s))}$ denotes the orthogonal projection to the kernel $N(A(s))$ of $A(s)$.
We say a crossing $s\in[0,1]$ is regular if its crossing operator is invertible.
We can perturb the paths $S^s(t)$ with fixed endpoints $s=0,1$ to make every crossing regular.
Then the difference of the Conley-Zehnder index of $\{\Phi^0(t)\}_{t\in[0,1]}$ and $\{\Phi^1(t)\}_{t\in[0,1]}$ is
\[
\CZ\Phi^1-\CZ\Phi^0=\frac{1}{2}\sign\Gamma(0,A)+\sum_{0<s<1}\sign\Gamma(s,A)+\frac{1}{2}\sign\Gamma(1,A),
\]
where $\sign$ denotes the signature (the number of positive eigenvalues minus the number of negative eigenvalues).
We will see later that the right hand side of the above equation only depends on $\{\Phi^0(t)\}_{t\in[0,1]}$ and $\{\Phi^1(t)\}_{t\in[0,1]}$.
We normalize $\CZ\Id=0$. Then we have a well-defined map $\CZ:\{\Phi(t):[0,1]\to\Sp(2n);\Phi(0)=\Id\}\to\frac{1}{2}\Z$.

The second definition uses the Maslov index of the path of Lagrangian subspaces.
First we recall its definition.
Fix a Lagrangian subspace $L_0\in \mathcal{L}=\{ L \subset (V,\omega);$ Lagranian$\}$ of a symplectic vector space $(V,\omega)$.
For a smooth path of Lagrangian subspaces $\alpha : [0,1]\to \mathcal{L}$,
$t\in[0,1]$ is called a crossing if
$\alpha(t) \in \Sigma_{L_0}=\{L\in\mathcal{L};L\cap L_0\neq 0\}$.
In this case we define the quadratic form on $\alpha(t)\cap L_0$ by 
$Q_t=\omega_0\circ (1\times (\partial_t \tilde{\alpha}(t) \circ \tilde{\alpha}(t)^{-1})):(\alpha(t)\cap L_0)\times(\alpha(t)\cap L_0)\to \R$,
where $\tilde{\alpha}$ is a path of Lagrangian frame
$\tilde{\alpha} : [0,1]\to \Hom_{\R} (\R^n ; V)$ such that $\alpha(t) = \mathrm{Im} \tilde{\alpha}(t)$. (The definition of $Q_t$ dose not depend on the choice of $\tilde{\alpha}$.)
We say a crossing $t\in[0,1]$ is regular if this quadratic form is nondegenerate.
We can perturb the path $\alpha$ with fixed endpoints $t=0,1$ to make every crossing regular.
Then the Maslov index of the path $\alpha$ is
\[
\mu_{L_0}(\alpha)=\frac{1}{2}\sign Q_0+\sum_{0<t<1}\sign Q_t+\frac{1}{2}\sign Q_1\ \in\frac{1}{2}\Z.
\]
This definition is independent of the perturbation.
The Maslov index satisfies the following properties:

\begin{itemize}
\item
homotopy invariance:
$\mu(\alpha)\in\frac{1}{2}\Z$ is determined by the fixed-endpoint homotopy class of $\alpha$.

\item
catenation:
If $\alpha,\beta:[0,1]\to \mathcal{L}$ satisfies $\alpha(1)=\beta(0)$, then
$\mu_{L_0}(\alpha\vee\beta)=\mu_{L_0}(\alpha)+\mu_{L_0}(\beta)$.

\item
direct sum:
For Lagrangian subspaces $L_0\subset(V,\omega)$ and $L'_0\subset(V',\omega')$, 
$\mu_{L_0\oplus L'_0}(\alpha\oplus\beta)=\mu_{L_0}(\alpha)+\mu_{L'_0}(\beta)$.

\item
symplectic invariance:
$\mu_{\Phi(L_0)}(\Phi\circ\alpha)=\mu_{L_0}(\alpha)$
for any linear symplectic isomorphism $\Phi:(V,\omega)\tocong (V',\omega')$.
\item
localization:
to state this property, we need some definitions.
Assume $V= L_0 \oplus L_1$, where $L_1$ is another Lagrangian.
Then it is easy to see that
$(V,\omega)=L_0\oplus L_1{\stackrel{1\times F}{\cong}}(L_0\oplus L_0^\ast,\omega_{L_0})$,
where
\[
\omega_{L_0}((x,\xi),(y,\eta))=\eta(x)-\xi(y)
\]
and
\[
F:L_1\tocong L_0^\ast \ v\mapsto -i_v\omega.
\]
It is also easy to see that every Lagrangian $L$ such that $L\cap L_1 = 0$
can be written as
\[
L=\mathrm{graph}_{L_0}^{L_1}q:=(1\times F)^{-1}\{(x,qx)\in L_0 \oplus L_0^\ast ; x\in L_0\},
\]
where $q\in S^2L_0^\ast$ is a quadratic form on $L_0$.

Localization property asserts that
if $\alpha$ has the form
\[
\alpha(t)=\mathrm{graph}_{L_0}^{L_1}q(t),
\]
where $q(t) : [0,1] \to S^2L_0^\ast $ is a path of quadratic forms on $L_0$, then
\[
\mu_{L_0}(\alpha )=\frac{1}{2}\{ \sign q(1)-\sign q(0)\}.
\]

\item
pseudo-continuity:
If $\beta:[0,1]\to\mathcal{L}$ is sufficiently close to $\alpha$ in $C([0,1];\mathcal{L})$, then
\[
|\mu_{L_0}(\beta)-\mu_{L_0}(\alpha)|\leq\frac{1}{2}(\dim(\alpha(0)\cap L_0)+\dim(\alpha(1)\cap L_0)).
\]
Moreover, if $\alpha(0)= \beta(0)$, then
\[
|\mu_{L_0}(\beta)-\mu_{L_0}(\alpha)|\leq\frac{1}{2}\dim(\alpha(1)\cap L_0).
\]
\end{itemize}

All the properties other than pseudo-continuity were proved in \cite{RS93}.
Pseudo-continuity can be proved as follows.
If $\beta$ is sufficiently close to $\alpha$ then 
there exists a homotopy $\alpha^s$ ($s\in [0,1]$) between $\alpha$ and $\beta$ such that 
under some decomposition $(V,\omega) = L_0 \oplus L_1$,
$\alpha^s(0) = \mathrm{graph}_{L_0}^{L_1} q_0(s)$ and 
$\alpha^s(1) = \mathrm{graph}_{L_0}^{L_1} q_1(s)$
for some quadratic forms $q_0(s)$ and $q_1(s)$.
Then by the properties of homotopy invariance, catenation, and localization imply
\begin{align*}
|\mu_{L_0}(\beta)-\mu_{L_0}(\alpha)|
&= |\mu _{L_0}^{L_1}(\alpha^s(1))_{s\in[0,1]} -  \mu _{L_0}^{L_1}(\alpha^s(0))_{s\in[0,1]}|\\
&= |\frac{1}{2}\{ ( \sign q_1(1) - \sign q_1(0) )  - ( \sign q_0(1) - \sign q_0(0) )\}| \\
&\leq \frac{1}{2}(\dim (\alpha (0)\cap L_0)+\dim(\alpha (1) \cap L_0))
\end{align*}
if $\beta(t)$ are sufficiently close to $\alpha(t)$. ($t=0,1$)

Using this index, we define for a path $\Phi(t):[0,1]\to \Sp(2n)$ with $\Phi(0)=\Id$,
\[
\CZ\Phi=\mu_\Delta((1\times \Phi(t))\Delta)\in\frac{1}{2}\Z,
\]
where $\Delta \subset (\R^{2n}\times \R^{2n},(-\omega_0)\oplus\omega_0)$ is
the diagonal Lagrangian and $(1\times \Phi(t))\Delta$ is the graph of $\Phi (t)$.
Note that pseudo-continuity of the Maslov index implies that of the Conley-Zhender index,
hence if $\Phi'$ is sufficiently close to $\Phi$ then
\[
|\CZ \Phi' -\CZ \Phi | \leq \frac{1}{2}\dim \Ker (\Phi (1)-\Id).
\]

The equivalence of the above two definitions was proved in Theirem 7.1 of \cite{RS95}.
Here is a sketch of its proof for reader's convenience. This contains the proof of
the well-definedness of the first definition. 

Let $\{S^s(t):[0,1]\to \gl(2n;\R)\}_{s\in[0,1]}$ be a family of paths of symmetric matrices
and $\{\Phi^s(t):[0,1]\to \Sp(2n)\}_{s\in[0,1]}$ be the fundamental solutions of the equations $\dot{x}=J_0S^s(t)x$.
Let $s_0 \in [0,1]$ be a regular crossing.
Then the quadratic form associated to the crossing operator is
\[
\xi  \mapsto \int^1_0 \langle \xi(t), \partial_s S^{s_0}(t) \xi(t) \rangle dt : \Ker (J_0\partial_t + S^{s_0}) \to \R.
\]
$\xi \in \Ker (J_0\partial_t + S^{s_0})$ can be written as $\xi(t) = \Phi^{s_0}(t) v$ with some $v\in \Ker(\Phi^{s_0}(1) -1)$.
Define $\xi(s,t) = \Phi^s(t) v$ and observe $\xi$ satisfies
\begin{gather*}
J_0 \partial_t \xi + S \xi = 0 \\
J_0\partial_s\partial_t\xi + \partial_s S \xi + S \partial_s \xi = 0.
\end{gather*}
Hence
\begin{align*}
\langle \xi, \partial_s S \xi \rangle
&=\langle J_0\xi, \partial_s \partial_t \xi \rangle - \langle \xi , S\partial_s \xi \rangle \\
&= \partial _t \langle J_0 \xi , \partial _s \xi \rangle,
\end{align*}
which implies that
\[
\int^1_0 \langle \xi(t), \partial_s S^{s_0}(t) \xi(t) \rangle dt 
= \langle J_0 \xi(s_0,1), \partial _s\xi(s_0,1) \rangle
= \langle J_0 \Phi^{s_0}(1)v, \partial _s\Phi^{s_0}(1)v \rangle.
\]
This coincides with the quadratic form associated with crossing of $(1\times \Phi^{s_0}(1))\Delta$
with $\Delta$.
Hence,
\begin{align*}
&\frac{1}{2}\sign\Gamma(0,A)+\sum_{0<s<1}\sign\Gamma(s,A)+\frac{1}{2}\sign\Gamma(1,A)\\
&= \mu_{\Delta}(1\times \Phi^s(1))_{s\in[0,1]}\\
&= \mu_{\Delta}(1\times \Phi^1(t))_{t\in[0,1]} - \mu_{\Delta}(1\times \Phi^0(t))_{t\in[0,1]}
\end{align*}
by catenation property and homotopy invariance of the Maslov index.
This proves the well-definedness of the first definition and the equivalence of the two definition.

\begin{defi}\label{max}
\[
\max \CZ \Phi = \CZ \Phi + \frac{1}{2}\mathrm{dim}\Ker (\Phi (1) - \Id )
\]
\end{defi}
\begin{rem}\label{maxrem}
For any small $\epsilon >0$, $\max \CZ \Phi = \CZ \{ e^{\epsilon J_0t}\Phi(t) \}_{t\in[0,1]}$. \end{rem}

The following lemma is classical in the theory of differential equations and
would be well known among the experts.
\begin{lem}[The Comparison Lemma]\label{comparison}
Let $S^0(t),S^1(t):[0,1]\to \gl(2n;\R)$ be two paths of symmetric matrices
and $\Phi^0(t),\Phi^1(0):[0,1]\to \Sp(2n)$ be the fundamental solutions of the equations $\dot{x}=J_0S^0(t)x$, $\dot{x}=J_0S^1(t)x$ respectively.
If $S^0(t)\leq S^1(t)$ for all $t\in[0,1]$, then $\CZ\Phi^0\leq \CZ\Phi^1$.
In particular,
\begin{enumerate}[\normalfont(i)]
\item if $S(t)\leq -C$ for all $t\in [0,1]$,where $C\in \R$ is a constant, then $\max \CZ\Phi\leq -n-2n\bigl[\frac{C}{2\pi}\bigr]^<$, where $[x]^<$ stands for the largest integer less than $x$.
\item if $S(t)\leq -C|_{V\times V} \oplus -\epsilon|_{V^{\bot}\times V^{\bot}}$ for all $t\in [0,1]$, where $C\in \R$ and $\epsilon >0$ are constants, $V \subset \R ^{2n}$ is a complex subspace and $V^{\bot}\subset \R^{2n}$ denotes its orthogonal compliment, then $\max \CZ \Phi \leq -n-2\mathrm{dim}_{\C}V\bigl[\frac{C}{2\pi}\bigr]^<$.
\end{enumerate}

\end{lem}

\begin{proof}
This is clear by the first definition since if we take $S^s(t)=(1-s)S^0(t)+sS^1(t)$ then the quadraic form associated with its crossing operator is
\[
(\Gamma(s,A)\xi,\xi)_{L^2}=\int_0^1\langle\xi,\partial_sS^s(t)\xi\rangle dt\qquad \xi\in N(A(s)),
\]
which is obviously nonnegative definite.
In the case (i),
\[
\CZ \{ e^{-J_0Ct}\} _{t \in [0,1]} \leq -n -2n \biggl[ \frac{C}{2\pi} \biggr]^<
\]
implies $\CZ \Phi \leq  -n -2n \bigl[ \frac{C}{2\pi} \bigr]^<$.
Since the right-hand side is upper semi-continuous with respect to $C$, Remark \ref{maxrem} implies the claim.
(ii) can be proved similarly.
\end{proof}

\begin{defi}

\begin{align*}
\Xi & =\{X\in \mathfrak{sp}(2n);\mathrm{Im} X\subset(\R^{2n},\omega_0) \mathrm{\ is\ isotropic}\}\\
     & =\{X\in \gl (2n;\R);\forall t\in \R \ 1+tX\in \Sp(2n)\}
\end{align*}

\end{defi}

\begin{rem}\label{parameterrem}
For a time-independent Hamiltonian $H:M\to \R$ and a smooth function
$\chi :\R \to \R$,
\[
(\phi ^{\chi \circ H}_t)_{\ast x}\xi = (\phi^H _{\chi' (H(x)) \cdot t})_{ \ast x}(\xi + t\chi''(H(x))\omega (\xi,X_H(x))X_H(x)):T_xM\to T_{\phi_t^{\chi \circ H}(x)}M.
\]
In this formula,
\[
X: \xi \mapsto \chi''(H(x))\omega (\xi,X_H(x))X_H(x) : T_xM \to T_xM
\]
belongs to $\Xi$.
\end{rem}

\begin{rem}\label{momentrem}
This remark will be used in the proof of Proposition \ref{caltorus}.
Consider a Hamiltonian $T^k$-action on a symplectic manifold $(M,\omega)$. Let $\mu : M \to \mathfrak{t}^\ast$ be its momentum map.
(We use the convention that for $Y\in \mathfrak{t}$, $Y\circ \mu : M\to \R$ is the Hamiltonian of the fundamental vector field $\underline{Y}$ of $Y$. ($\underline{Y}_x= \frac{d}{dt}e^{tY} \cdot x|_{t=0}$.))
For a smooth function $f : \mathfrak{t}^\ast \to \R$, we regard its differential as $df : \mathfrak{t}^\ast \to \mathfrak{t}$.

Consider the Hamiltonian $H= f\circ \mu : M \to \R$.
It is easy to see
\begin{gather*}
\phi^H_t(x) = e^{tdf(\mu (x))}\cdot x\ :M\to M\\
(\phi^H_t)_{\ast x} \xi =(e^{tdf(\mu (x))})_{\ast x} \bigl(\xi + t\underline{(df\circ \mu)_{\ast}\xi}_x \bigr)\ :T_xM \to T_{\phi^H_t(x)}M.
\end{gather*}
In this formula,
\[
X : \xi \mapsto \underline{(df\circ \mu)_{\ast}\xi}_x \ :T_xM \to T_xM
\]
belongs to $\Xi$.

\end{rem}

\begin{lem}\label{time}
For $\Phi(t):[0,1]\to \Sp(2n)$ with $\Phi(0)=\Id$ and $X\in \Xi$,
\begin{align*}
|\CZ (\Phi(t)\circ (1+tX))-\CZ (\Phi(t))|&\leq \rank X\\
|\max \CZ (\Phi(t)\circ (1+tX))-\max \CZ (\Phi(t))|&\leq \rank X\\
\max \CZ (\Phi(t)\circ (1+ tX)) - \max \CZ (\Phi(t)) \qquad&\\
\leq \rank X - \bigl( \rank (\Phi(1)\circ(1 + X)-\Id) &- \rank (\Phi(1) -\Id) \bigr)
\end{align*}
\end{lem}
\begin{proof}
First we show the first inequality. Put $A = \Phi(1) \in \Sp(2n)$.
Then by the catenation property and homotopy invariance of the Maslov index, the left hand side of the inequality coincides with 
\[
|\mu_{\Delta}((1\times A(1+tX))\Delta)_{t\in[0,1]}|.
\]
Take a Lagrangian subspace $L_0 \subset (\R^{2n},\omega_0)$ containing
$\mathrm{Im} AX$, and let $L_1 \subset (\R^{2n},\omega_0)$ be another Lagrangian subspace such that $L_1\oplus L_0=\R^{2n}$ and $L_0\cap AL_1=0$.
Then under the decomposition $\R^{2n}\times\R^{2n}=\Delta\oplus (L_1\times L_0)$,
\[
(1\times A(1+tX))\Delta = \mathrm{graph}_\Delta^{L_1\times L_0}(q_0+tq_1),
\]
where $q_0$ and $q_1$ are some quadratic forms on $\Delta$, and $\rank q_1=\rank X$.
Indeed,
\begin{align*}
(x,A(1 + tX)x) =& \{((P_0 + P_1 A)x, (P_0 + P_1 A)x)\} \\
&+ \{(P_1 (1-A)x, P_0(A-1)x) + t (0, AXx)\}\\
&\in \Delta \oplus (L_1\times L_0) \quad \text{ for any } x\in \R^{2n}
\end{align*}
and $P_0 + P_1A$ is invertible, where $P_0$ and $P_1$ stands for the projections to $L_0$ and $L_1$ respectively.
Hence by the localization property of the Maslov index,
\[
\mu_{\Delta}((1\times A(1+tX))\Delta)_{t\in[0,1]}=\frac{1}{2} (\sign(q_0+q_1)-\sign q_0).
\]
The claim follows by the fact that for symmetric matrices $A$ and $B$,
\[
|\sign (A+B)-\sign A|\leq 2\rank B.
\]

The second inequality is a corollary of the first one and Remark \ref{maxrem}. 

In the above setting, the left-hand side of the third inequality is equal to
\[
\frac{1}{2} ( \signn (q_0 + q_1) - \signn q_0 ),
\]
where $\signn q = p_+ (q) + p_0 (q) - p_-(q)$. ($p_+ (q)$, $p_0 (q)$, $p_- (q)$ are the number of positive, zero, minus eigenvalues respectively.)
The claim follows by the fact that if symmetric matrices $A$ and $B$ satisfy
$\rank (A+B) \geq \rank A + k$, then
\[
\signn (A+B)-\signn A\leq 2\rank B - 2k,
\]
This is because
\begin{gather*}
p_0(A + B) - p_0(A) \leq -k,\\
p_+(A+B) - p_+(A) \leq \rank B, \text{ and}\\
\signn (A+B) - \signn A = 2\bigl( (p_+ + p_0)(A+B) - (p_+ + p_0)(A) \bigr).
\end{gather*}
\end{proof}

\begin{cor}\label{parameter}
Under the condition of Remark \ref{parameterrem},
we fix a symplectic trivialization $T_xM\cong T_{\phi_t^{\chi \circ H}(x)}M$.
Then
\[
|\max\CZ ((\phi ^{\chi \circ H}_t)_{\ast x})_{t\in[0,1]} - \max \CZ ((\phi ^H_{\chi'(H(x))\cdot t})_{\ast x})_{t\in[0,1]}|\leq 1.
\]
\end{cor}

\begin{cor}\label{moment}
Under the condition of Remark \ref{momentrem},
\[
\max \CZ ((\phi^H_t)_{\ast x})_{t\in[0,1]} \leq \max \CZ ((e^{tdf(\mu (x))})_{\ast x})_{t\in[0,1]}
\]
for every $x\in \Fix \phi_1^H$ and any symplectic trivialization of $TM$ over the closed orbit.
\end{cor}
\begin{proof}
For $x\in \Fix \phi_1^H$,
\[
\rank \bigl( (e^{df(\mu (x))})_{\ast x}(1+X) -\Id \bigr) - \rank \bigl( (e^{df(\mu (x))})_{\ast x} - \Id \bigr) = \rank X,
\]
since $\mathrm{Im}((e^{df(\mu (x))})_{\ast x}-\Id) \cap \mathrm{Im}X = 0$ and $(e^{df(\mu (x))})_{\ast x} X = X$.	
Hence Lemma \ref{time} implies the claim.
\end{proof}

\section{Proof of the main theorem and examples}\label{proof}
First we prove some well-known properties of convex open subsets of $\R^{m}$
needed for the proof of the main theorem.
Recall we say a bounded open subset $U\subset \R^{m}$ is strictly convex if there exists
a smooth function $F:\R^{m}\to \R$ such that $D^2F>0$ and $\partial U =\{F=0\}$.
\begin{lem}\label{conv}
Let $U\subset \R^m $ be a convex open subset. Then for any compact subset $K\subset U$,
there exists a strictly convex open subset $K\subset V\Subset U$.
\end{lem}
\begin{proof}
We may assume $0\in U$.
Define
\[
p(x)=\inf \{t>0; t^{-1}x\in U\} :\R^m \to \R_{\geq 0}, \ f(x)=p(x)^2 :\R^m \to \R_{\geq 0}.
\]
Then
\[
f((1-t)x + ty)\leq (1-t)f(x) + tf(y) \ 0\leq t \leq1.
\]
Let $f_\delta = \phi_\delta \ast f$ be the convolution with a mollifier $\phi_\delta(x)=\delta^{-m}\phi (\frac{x}{\delta})$, where $\phi$ is a nonnegative smooth function with compact support whose integral is one.
Then by the above inequality and the definition of convolution, $f_\delta$ also satisfies
\[
f_\delta ((1-t)x + ty)\leq (1-t)f_\delta (x) + tf_\delta (y) \ 0\leq t \leq1,
\]
which implies $D^2f_\delta \geq 0$.
It follows for appropriate $\delta>0$ and $\epsilon>0$,
$V_{\delta, \epsilon}=\{ f_\delta + \epsilon |x|^2 < 1 \} \subset U$ is the required subset since $D^2(f_\delta + \epsilon |x|^2-1)>0$.
\end{proof}

\begin{lem}\label{strictconv}
Let $U\subset \R^m$ be a strictly convex open subset such that $0\in U$.
Define $f: \R^m \setminus 0 \to \R$ by
$f(t\cdot y)=t^2$ for $t>0$ and $y\in \partial U$.
Then its Hessian $D^2f(x)$ is positive definite on $\R^m\setminus 0$.
\end{lem}
\begin{proof}
Since $D^2f(a\cdot x) = D^2f(x)$ for $a>0$, it is enough to show $D^2f(x)>0$ for $x\in \partial U$.
Take any $v\in T_x\partial U$. Since $f$ is defined by the equation $F(f(y)^{-\frac{1}{2}}y)=0$,
$F( f(x + tv)^{-\frac{1}{2}} (x+tv))=0$ for all $t\in\R$.
Differentiating this equation at $t=0$, we get
\[
DF\cdot v =\frac{1}{2}(DF\cdot x)(Df\cdot v).
\]
Hence $Df\cdot v=0$ because $DF\cdot v =0$.
Differentiating the same equation twice at $t=0$ and using $Df\cdot v=0$, we get
\[
D^2F(v,v)=\frac{1}{2}(DF\cdot x)D^2f(v,v),
\]
which implies $D^2f(x)(v,v)>0$.
It is also easily seen that $D^2f(x)(x,x)=2>0$ and $D^2f(x)(x,v)=\frac{\partial ^2}{\partial s \partial t}f(s(x+tv))\bigr| _{s=1,t=0}=0$.
\end{proof}

For a function $H : \R^{2n} \to \R$, define $A_H : \R^{2n} \to \R$ by
\[
A_H (x)=-\int_{\phi^H_{[0,1]}(x)} \lambda + H(x),
\]
where $\lambda = \frac{1}{2} \sum_{j=1}^n (x_j dy_j - y_j dx_j)$.

\begin{lem}\label{delta}
Let $U\subset R^{2n}$ be a strictly convex open subset such that $0\in U$ and define $f: \R^{2n} \setminus 0 \to \R$ as above.
Then for its convolution $f_{\delta}=f\ast \phi_{\delta}$ with a mollifier $\phi_{\delta}$,
\begin{enumerate}[\normalfont(i)]
\item
if $D^2f \geq a $ on $\R^{2n}\setminus 0$, then $D^2f_{\delta} \geq a$ on $\R^{2n}$.
\item
there exist some constant $C>0$ such that
\[
\bigl| A_{\chi\circ f_{\delta}} (x) - \bigl( -\chi' (f_{\delta}(x))f_{\delta}(x) + \chi (f_{\delta}(x)) \bigr) \bigr| \leq C \delta |\chi'|_{L^{\infty}}
\]
for any $\chi :\R \to \R$ and $x \in U$.
\end{enumerate}
\end{lem}
\begin{proof}
(i)
The assumption implies $f - \frac{a}{2}|x|^2$ is convex on $\R^{2n}$.
Therefore $(f - \frac{a}{2}|x|^2)\ast \phi_{\delta}$ is also convex, which implies the claim.

(ii)
Notice for any function $H$, $i_{X_H} \lambda (x) = \frac{1}{2} DH(x)\cdot x$. This implies
\begin{align*}
\int_{\phi^{\chi\circ  f_{\delta}} _{[0,1]} (x) } \lambda &= \int^1_0 i_{X_{\chi \circ f_{\delta}}} \lambda (\phi^{\chi \circ f_{\delta}}_t (x)) dt\\
&=\frac{1}{2} \chi' (f_{\delta}(x)) \int ^1_0 Df_{\delta} (\phi^{\chi\circ f_{\delta}}_t (x)) \cdot \phi^{\chi \circ f_{\delta}}_t (x)  dt  \\
\end{align*}
Hence
\begin{align*}
&\int_{\phi^{\chi\circ  f_{\delta}} _{[0,1]} (x) } \lambda - \chi' (f_{\delta}(x))f_{\delta}(x) \\
&= \frac{1}{2} \chi' (f_{\delta}(x)) \int ^1_0 \Bigl( Df_{\delta} (\phi^{\chi\circ f_{\delta}}_t (x)) \cdot \phi^{\chi \circ f_{\delta}}_t (x) - 2f_{\delta}(\phi^{\chi\circ f_{\delta}}_t (x)) \Bigr)dt.
\end{align*}
By the definition of the convolution,
\begin{gather*}
Df_{\delta} (\phi^{\chi\circ f_{\delta}}_t (x)) \cdot \phi^{\chi \circ f_{\delta}}_t (x)
= \int _{\R^{2n}} \phi_{\delta} (y) Df( \phi^{\chi \circ f_{\delta}}_t (x) - y ) \cdot \phi^{\chi \circ f_{\delta}}_t (x) dy, \\
2f_{\delta}(\phi^{\chi\circ f_{\delta}}_t (x)) = \int _{\R^{2n}} \phi_{\delta} (y) Df( \phi^{\chi \circ f_{\delta}}_t (x) - y ) \cdot (\phi^{\chi \circ f_{\delta}}_t (x) - y )dy,
\end{gather*}
where we have used that $2f(x) = Df(x) \cdot x$.
These two equations imply
\begin{align*}
\bigl|Df_{\delta} (\phi^{\chi\circ f_{\delta}}_t (x)) \cdot \phi^{\chi \circ f_{\delta}}_t (x)
-
2f_{\delta}(\phi^{\chi\circ f_{\delta}}_t (x))\bigr|
&= \biggl|\int _{\R^{2n}} \phi_{\delta} (y) Df( \phi^{\chi \circ f_{\delta}}_t (x) - y ) \cdot y dy\biggr| \\
& \leq C|Df|_{L^{\infty}(2U)} \delta,
\end{align*}
where we assume the support of $\phi$ is contained in the ball of radius $C$. This proves the assertion.

\end{proof}

\begin{defi}
For a strictly convex open subset $U\subset \R^{2n}$ such that $0\in U$,
\begin{itemize}
\item
$\widehat{C}(U)=\frac{2\pi}{a}$, where $a= \min_{x\in \R^{2n}\setminus 0}$ \{the minimal eigenvalue of 
$D^2f(x)$\}
\item
$C(U)=\inf$\{$\widehat{C}(V) ; V \subset \R^{2n}$ strictly convex and $(V,\omega_0)\cong (U,\omega_0) $\}
\item
$\widehat{C}_0(U) = \inf_{V \subset \R^{2n}\setminus 0, a>0}
\frac{2\pi}{a}$,
where infimum is taken over all one-dimensional complex subspace $V\subset \R^{2n}$ and $a>0$
such that $D^2f(x) > a|_{V\times V} \oplus 0|_{ V^{\bot} \times V^{\bot} }$
for every $x\in \R^{2n}\setminus 0$.
\item
$C_0(U)=\inf$\{$\widehat{C}_0(V) ; V \subset \R^{2n}$ strictly convex and $ (V,\omega_0)\cong (U,\omega_0)$\}
\end{itemize}
\end{defi}
We note that in the definition of $\widehat{C}(U)$,
minimum over $R^{2n}\setminus 0$ is obtained since $D^2f(ax)= D^2f(x)$ for all $a>0$.

For example,
if $U=E(r_1,r_2,\dots,r_n)=\{ z=(z_1,z_2,\dots,z_n)\in \R^{2n}; \sum \frac{|z_i|^2}{r^2_i} < 1\}$ $(r_1\leq r_2\leq \dots \leq r_n)$ is an ellipsoid,
then $f = \sum \frac{|z_i|^2}{r^2_i}$, therefore
$\widehat{C}(U)=\pi r_n^2$ and $\widehat{C}_0(U)=\pi r_1^2$.

Theorem \ref{main} is a corollary of Lemma \ref{conv} and the following Proposition.
\begin{prop}\label{bound}
Let $(M,\omega)$ be a closed symplectic manifold satisfying $c_1 = \kappa \omega$ on
$\pi_2(M)$.
Assume strictly convex open subsets $U_j \subset (\R^{2n},\omega_0)$ are disjointly symplectically embedded in $(M,\omega)$.
Then for any function $F\in C(M\times S^1)$ such that $F|_{(M\setminus \coprod U_j)\times S^1}=0$,
\begin{itemize}
\item
if $\kappa \leq 0$ then
\[
0\leq c([M];F) \leq \max_j C_0(U_j)
\]
\item
if $\kappa > 0$ and $\max C(U_j) \leq \frac{n}{\kappa}$, then
\[
0\leq c([M]; F) \leq \max_j C(U_j)
\]

\end{itemize}

\end{prop}
\begin{proof}
The proof is based on some calculations on $\R^{2n}$.
Note for a Hamiltonian $H:\R^{2n} \to \R$,
the differential $(\phi^H_t)_\ast$ of the Hamiltonian flow is determined by the differential equations
\[
\frac{d}{dt}(\phi^H_t)_{\ast x}\xi =J_0D^2H(\phi^H_t(x))\cdot (\phi^H_t)_{\ast x}\xi\quad \forall\xi\in\R^{2n},\quad (\phi^H_0)_{\ast x}=\Id,
\]
where we regard $(\phi^H_t)_\ast$ as matrices under the natural trivialization of the tangent bundle of $\R^{2n}$.
This implies that along the orbit $\phi^H_t(x)$, $(\phi^H_t)_{\ast x}$ is the fundamental solution of the corresponding differential equation, hence we can apply the consequence of Section \ref{CZ}.

First we show the case $\kappa =0$.
For each $U_j \subset \R^{2n}$, define $f_j : \R^{2n}\setminus 0 \to \R$ as in the  Lemma \ref{strictconv}
and assume $D^2f_j >a|_{V_j\times V_j} \oplus 0|_{V_j^{\bot} \times V_j^{\bot}}$.
Fix arbitrary small constant $\epsilon>0$. Take $\epsilon' > 0$ sufficiently small.
Let $\chi_0, \chi_1 : [0,1] \to \R_{\geq 0}$ be monotone decreasing functions
such that
\begin{itemize}
\item
$\max \chi_0 <\frac{2\pi}{a}$, $\chi'_0 \equiv -(\frac{2\pi}{a} + \epsilon')$ on $[0,1-\epsilon)$, $\supp \chi_0 \subset [0,1)$
\item
$\chi_1$ is linear on $[0,1-2\epsilon]$, $\supp \chi_1 \subset [0,1-\epsilon)$.
\end{itemize}
Define $\chi^s = \chi_0 + s\chi_1 : [0,1] \to \R_{\geq 0}$ $(s\in [0,\infty))$ and consider a family of Hamiltonians
$H^s = \sum \chi^s \circ (f_j)_\delta : \coprod U_j \to \R$, which we can extend on $M$ by zero if $\delta > 0$ is sufficiently small.

For $x\in \{ (f_j)_{\delta} \leq 1-\epsilon \}$,
Corollary \ref{parameter} implies
\[
\max \CZ ((\phi^{H^s}_t)_{\ast x})_{t\in [0,1]} \leq \max \CZ ((\phi^{(f_j)_\delta}_{{\chi^s}'((f_j)_{\delta}(x)) t})_{\ast x})_{t\in [0,1]} + 1.
\]
By Lemma \ref{comparison} (ii),
\[
\max \CZ ((\phi^{(f_j)_\delta}_{{\chi^s}'((f_j)_{\delta}(x)) t})_{\ast x})_{t\in [0,1]} + 1
\leq-n-2+1 < -n
\]
since $a{\chi^s}'((f_j)_{\delta}(x)) < -2\pi$.
Hence
\[
\max \CZ ((\phi^{(f_j)_\delta}_{{\chi^s}'((f_j)_{\delta}(x)) t})_{\ast x})_{t\in [0,1]} + 1 < -n.
\]
On the other hand, $H^s=H^0$ on $M\setminus \coprod \{ (f_j)_{\delta} \leq 1-\epsilon \}$.

Let $G^s : M\times S^1 \to \R$ be small perturbations of $H^s$ such that $G^s\equiv G^0$ on
$M\setminus \coprod \{ (f_j)_{\delta} \leq 1-\epsilon \}$, and
$G^s$ is non-degenerate for any $s\in [0,\infty)\setminus A$, where $A\subset [0,\infty)$ is some countable subset.

Let $x : S^1 \to M$ be a periodic orbit of $G^s$.
If $x(t_0) \in \coprod \{ (f_j)_{\delta} \leq 1-\epsilon \}$ for some $t_0\in S^1$, then
$\CZ _{G^s}x < -n$ provided that $G^s$ is sufficiently close to $H^s$.
Hence periodic orbits which have the Conley-Zehnder index $-n$ are all contained in
$M\setminus \coprod \{ (f_j)_{\delta} \leq 1-\epsilon\}$ and are independent of $s$, which implies
$\Spec _nG^s = \Spec _nG^0$ for every $s\in [0,\infty)\setminus A$.
(See (\ref{spec_k}) in Section \ref{review} for the definition of $\Spec _k$.)
By the spectrality and Lipschitz continuity of the spectral invariant,
we have $c([M] ; G^s) = c([M] ; G^0)$.
We conclude that
\[
c([M] ; H^s) = c([M] ; H^0) \leq \max_M H^0 \leq \frac{2\pi}{a}.
\]

For any function $F\in C(M\times S^1)$ such that $F|_{(M\setminus \coprod U_j)\times S^1}=0$ and any $\epsilon'' > 0$,
there exist a small $\epsilon > 0$ such that
$\{ F\geq \epsilon''\} \subset \coprod \{f_j \leq 1-4\epsilon \}$.
We take $\delta>0$ sufficiently small so that $\{ F\geq \epsilon'' \} \subset \coprod \{(f_j)_\delta \leq 1-3\epsilon \}$.
Take $H^s$ as above for these $\epsilon>0$ and $\delta>0$, then
there exists $s>0$ such that $F\leq \epsilon'' + H^s$.
Hence by the monotonicity,
\[
c([M];F) \leq \epsilon'' + c([M];H^s) \leq \epsilon'' + \frac{2\pi}{a}.
\]
We conclude $c([M];F)\leq \max_j C_0(U_j)$.
In particular, $M\setminus \coprod U_j$ is $[M]$-superheavy, which implies $c([M];F)\geq 0$ by Lemma \ref{geq0}.

Next we consider the case $\kappa <0$.
For each $U_j \subset \R^{2n}$, define $f_j : \R^{2n} \setminus 0\to \R$ as in Lemma \ref{strictconv}
and assume $D^2f_j >a|_{V_j\times V_j} \oplus 0|_{V_j^{\bot} \times V_j^{\bot}}$.
Fix arbitrary small constant $\epsilon>0$. Take $\epsilon' > 0$ sufficiently small.
Define $\chi^s : [0,1] \to \R_{\geq 0}$ $(s\in [0,\infty))$ as in the case $\kappa = 0$
and consider a family of Hamiltonians
$H^s = \sum \chi^s \circ (f_j)_{\delta} : \coprod U_j \to \R$, which we can extend on $M$ by zero if $\delta > 0$ is sufficiently small.

Define $\widehat{A}_{H^s} : \coprod U_j \to \R$ by
\[
\widehat{A}_{H^s} (x) = -\int _{\phi ^{H^s}_{[0,1]}(x)} \lambda  + H^s(x)
+ \frac{1}{\kappa}\Bigl[ \frac{n + \max \CZ \phi^{H^s}_{t\ast x}}{2} \Bigr],
\] 
where $[x]$ stands for the largest integer less than or equal to $x$.
By Corollary \ref{parameter} and Lemma \ref{comparison} (ii),
\begin{align*}
n + \max \CZ (\phi^{H^s}_t)_{\ast x} &\leq n + \max \CZ (\phi ^{(f_j)_{\delta}} _{{\chi^s}'((f_j)_{\delta}(x)) t})_{\ast x}  + 1\\
&\leq -2\biggl[\frac{-a{\chi^s}'((f_j)_{\delta}(x))}{2\pi}\biggr]^< + 1,
\end{align*}
which implies
\[
\widehat{A}_{H^s} (x) \geq -{\chi^s}'((f_j)_{\delta}(x)) (f_j)_{\delta}(x) + \chi^s ((f_j)_{\delta}(x)) - C\delta |{\chi^s}'|_{\infty} -\frac{1}{\kappa}\biggl[\frac{-a{\chi^s}'((f_j)_{\delta}(x))}{2\pi}\biggr]^<
\]
on $U_j$, where we have used Lemma \ref{delta}.

For $s\geq 0$ and $x\in \{ (f_j)_{\delta} \leq 1-\epsilon \}$,
\begin{align*}
\widehat{A}_{H^s}(x) &\geq \max H^0 - \frac{1}{\kappa} - C\delta |{\chi^s}'|_{\infty}\\
&> \max H^0
\end{align*}
if $\delta >0$ is sufficiently small for given $s\geq 0$.
(More precisely, first we fix some large $T>0$, then there exists some $\delta>0$ such that
the above inequality holds for every $s\in [0,T]$.) Take $\epsilon'' >0$ such that
$\widehat{A}_{H^s}(x) > \max H^0 + 2\epsilon''$.

Let $G^s_t : M\times S^1 \to \R$ be small perturbations of $H^s$
such that $G^s=G^0$ on $M \setminus \coprod \{ (f_j)_{\delta} \leq 1-\epsilon \}$
and every one-periodic orbit is nondegenerate for any $s\in [0,\infty) \setminus A$, where $A\subset [0,\infty)$ is some countable subset.
If $[x,u]\in \Crit \mathcal{A}_{G^s}$ has the Conley-Zehnder index $-n$
and $x(t_0) \in \{ (f_j)_{\delta} \leq 1-\epsilon \}$ for some $t_0\in S^1$,
then $[x,u] = x\cdot A$ with 
\begin{align*}
2c_1(A) &= n + \CZ (\phi^{G^s}_t)_{\ast x(0)}\\
&= n + \CZ (\phi^{G^s}_{t_0})_{\ast x(0)}(\phi^{G^s}_t)_{\ast x(0)} ((\phi^{G^s}_{t_0})_{\ast x(0)})^{-1}\\
&\leq n + \max \CZ (\phi^{H^s}_t)_{\ast x(t_0)}.
\end{align*}
Hence
\[
\mathcal{A}_{G^s} ([x,u]) = A_{G^s} (x) + \frac{1}{\kappa} c_1(A) \geq \widehat{A}_{H^s}(x(t_0)) -\epsilon'' > \max H^0 + \epsilon'',
\]
provided that $G^s$ is sufficiently close to $H^s$.
On the other hand, if $x \subset M \setminus \coprod \{ (f_j)_{\delta} \leq 1-\epsilon \}$,
then $\mathcal{A}_{G^s} ([x,u]) = \mathcal{A}_{G^0} ([x,u])$ for every $s \geq 0$.
We conclude that
\[
\Spec _n G^s \cap (-\infty ,\max H^0 + \epsilon'' ] = \Spec _n G^0 \cap (-\infty ,\max H^0 + \epsilon'' ] 
\]
for $s \in [0,\infty)\setminus A$,
which implies $c([M];G^s) = c([M];G^0)$.
Hence $c([M];H^s) = c([M];H^0)$.
The rest of the proof continues in the same way as in the case
$\kappa = 0$.

Finally we prove the case $\kappa >0$.
For each $U_j \subset \R^{2n}$, define $f_j : \R^{2n} \to \R$ as in the  Lemma \ref{strictconv},
and assume $D^2f_j \geq a >0$ and $\frac{2\pi}{a} \leq \frac{n}{\kappa}$.
Let $\chi :[0,1] \to \R_{\geq 0}$ be an arbitrary monotone decreasing function such that
$\chi'' \geq 0$ on $[0,1]$ and $\mathrm{supp} \chi \subset [0,1)$.
Consider a Hamiltonian $H= \chi \circ (f_j)_{\delta}$.
By Corollary \ref{parameter} and Lemma \ref{comparison} (i),
\begin{align*}
n + \max \CZ (\phi^H_t)_{\ast x} &\leq n + \max \CZ (\phi ^{(f_j)_{\delta}} _{\chi'((f_j)_{\delta}(x)) t})_{\ast x}  + 1\\
&\leq -2n\biggl[\frac{-a\chi'((f_j)_{\delta}(x))}{2\pi}\biggr]^< + 1.
\end{align*}
Hence if $H(x)\neq 0$ (that is, if $\chi'(f_j(x))\neq 0$),
\begin{align*}
\widehat{A}_H (x) &\leq -\chi'((f_j)_{\delta}(x))(f_j)_{\delta}(x) + \chi((f_j)_{\delta}(x))
- \frac{n}{\kappa}\biggl[\frac{-a\chi'((f_j)_{\delta}(x))}{2\pi}\biggr]^< + C\delta |\chi'|_{\infty}\\
&\leq -\chi'((f_j)_{\delta}(x)) - \frac{n}{\kappa}\biggl[\frac{-a\chi'((f_j)_{\delta}(x))}{2\pi}\biggr]^< + C\delta |\chi'|_{\infty} \qquad \quad \text{(since } \chi'' \geq 0\text{ )} \\
&\leq \frac{2\pi}{a} + C\delta |\chi'|_{\infty},
\end{align*}
where we have used that $y - \frac{n}{\kappa}\bigl[ \frac{ay}{2\pi} \bigr]^< \leq \frac{2\pi}{a}$ for $y > 0$. (This is a consequence of the assumption $\frac{2\pi}{a}\leq \frac{n}{\kappa}$.)

Let $G_t : M\times S^1 \to \R$ be a small perturbation of $H$.
Suppose $[x,u]\in \Crit \mathcal{A}_{G^s}$ has the Conley Zehnder index $-n$. 
If $x$ is contained in a small neighborhood $N$ of $\{ H=0 \}$,
then $\mathcal{A}_G([x,u]) \in \frac{1}{\kappa}\Z + [-\epsilon, \epsilon]$.
If not, then
$x(t_0) \in U_j\setminus N $ for some $t_0\in S^1$ and $j$,
and $[x,u] = x\cdot A$ with $2c_1(A) = n + \CZ (\phi^G_t)_{\ast x(0)} \leq n + \max \CZ (\phi^H_t)_{\ast x(t_0)}$.
Hence
\[
\mathcal{A}_G([x,u]) \leq \widehat{A}_H (x(t_0)) + \epsilon \leq \frac{2\pi}{a} + C\delta |\chi'|_{\infty} + \epsilon,
\]
where $\epsilon>0$ can be made arbitrary small if $G$ is sufficiently close to $H$.

We conclude that $\Spec_n G \subset (-\infty , \frac{2\pi}{a} + C\delta |\chi'|_{\infty} + \epsilon] \cup (\frac{1}{\kappa}\Z + [-\epsilon, \epsilon])$. Since $c([M];G)$ is contained in this set and $G$ can be taken arbitrary close to $H$,
$c([M];H) \in ( -\infty, \frac{2\pi}{a} + C\delta |\chi'|_{\infty}] \cup \frac{1}{\kappa}\Z$.
Replacing $\chi$ by $s\chi$ $(0\leq s \leq 1)$,
we see $c([M];sH) \in ( -\infty, \frac{2\pi}{a} + C\delta |\chi'|_{\infty}] \cup \frac{1}{\kappa}\Z$.
Therefore $c([M];H) \leq \frac{2\pi}{a} + C\delta |\chi'|_{\infty}$ by the Lipschitz continuity of spectral invariant. Since $\chi$ and $\delta$ are arbitrary,
we can continue the proof in the same way as in the case $\kappa =0$.
\end{proof}

\begin{rem}\label{variations}
There are some variations of the above proposition.
For example, if $\kappa > 0$ and $\max C_0(U_j) \leq \frac{1}{\kappa}$, then
\[
0\leq c([M]; F) \leq \max_j C_0(U_j).
\]
The above argument can be used for the product $U= N\times V$ of a closed sympelctic manifold $(N,\omega')$ and a strictly convex subset $V\subset (\R^{2m},\omega_0)$. In this case we can use a Hamiltonian $H$ or $H^s$ which is independent of $y\in N$ in the above proof.
\end{rem}

\begin{proof}[Proof of Theorem \ref{main}]
Take $H_0(x) = \dist (x,M\setminus\coprod U_j)$.
For any $\epsilon >0$, there exist strictly convex open subsets $V_j\subset U_j$ such that
$\supp (H_0-\epsilon)_+ \subset \coprod V_j$. ($()_+$ denotes positive part of the function.)
Hence $\zeta([M];H_0)\leq \zeta([M];(H_0-\epsilon)_+) + \epsilon \leq \epsilon$,
which implies $M\setminus \coprod U_j$ is $[M]$-superheavy.
For a general non-zero idempotent $\alpha$, triangle inequality of the spectral invariant implies
$\zeta (\alpha; H_0) \leq \zeta ([M];H_0)$, hence $M\setminus \coprod U_j$ is $\alpha$-superheavy.
\end{proof}

\begin{eg}
Let $(\Sigma _g , \omega)$ be a Riemann surface of genus $g\geq 1$ and
$\Sigma _g = e^0 \cup e^1_1 \cup e^1_2 \cup \dots \cup e^1_{2g} \cup e^2$ be
its CW-decomposition.
Then the $S^1$ bouquet
$\bigvee_1^{2g} S^1 = e^0 \cup e^1_1 \cup e^1_2 \cup \dots \cup e^1_{2g}$ is
$[\Sigma _g]$-superheavy.
This is a minimal superheavy subset since any non-contractible loop in $\Sigma _g$ is
$[\Sigma_g]$-heavy (this can be easily seen by direct calculation).
The case $g=1$ was proved by M. Kawasaki in \cite{Ka14} and the general case was
proved by V. Humili\`{e}re, F. Le Roux, S. Seyfaddini in \cite{HRS}.
\end{eg}

\begin{eg}
Let $T=(\R^{2n}/\Gamma ,\omega_0)$ be a torus, where $\Gamma = \bigoplus _{k=1}^{2n} \Z w_k \subset \R^{2n}$ be a lattice.
If $\mathrm{span}_\R \{w_1, w_2 \}$ is symplectic, then \{$\sum _{k=1}^{2n} t_k w_k \in T; t_1=0$ or $t_2=0$\} is $[T]$-superheavy.
Indeed, we may assume $w_2 = J_0 w_1$ since there exists a symplectic transform $A\in \Sp (2n) $ such that $Aw_2 = J_0 w_1 $. Define
\begin{align*}
\pi : U=\{ \sum _{k=1}^{2n} t_k w_k \in T; 0<t_1,t_2<1\}
&\to V = \{ \sum _{k=1}^2 t_k w_k \in \R^{2n} ; 0<t_1,t_2<1\}\\
\sum _k t_k w_k &\mapsto \sum _{k=1}^2 t_k w_k.
\end{align*}
For any strictly convex open subset $V_0\subset V$, define $f_{V_0} : \C w_1 \to \R$
as in Lemma \ref{strictconv}.
Then $f = f_{V_0} \circ \pi : \pi^{-1}(V_0) \to \R$ satisfies $D^2f \geq \epsilon |_{\C w_1\times \C w_1} + 0|_{\C w_1^{\bot} \times \C w_1^{\bot}}$ for some $\epsilon>0$.
Hence the same reasoning as in the proof of Theorem \ref{main} shows the claim.

We note if $\{t_1=0$ or $t_2 =0 \}$ and $\{t_3=0$ or $t_4 =0 \}$ are both superheavy, and if
$\{t_i,t_j\} = 0$ for $i=1,2$, $j= 3,4$, then $\{t_1=0$ or $t_2 =0 \} \cap \{t_1=0$ or $t_2 =0 \}$ is also  
superheavy by Lemma \ref{com}.

\end{eg}

\begin{eg}\label{counter}
Let $(\C P^n,\tau _0)$ be the complex projective space with the Fubini-Study form.
Then it is easy to see that $C(B) = \frac{n}{\kappa}$, where
$B = \{ [z_0:z_1:\dots:z_n] \in \C P^n ; \frac{|z_n|^2}{\sum |z_j|^2} > \frac{1}{n+1}\}$.
Since $C=\{ |z_0|=|z_1|=\dots=|z_n| \} \subset \C P^n$ is superheavy (see section \ref{another}),
the above proposition dose not hold for the ball which contain $C$ by Lemma \ref{heavysuperheavy}.
See also Example \ref{proj}
\end{eg}

\section{An application to Poisson bracket invariants}\label{application}
As mentioned earlier, we can extend Theorem 9 of \cite{Sey14} by our Proposition \ref{bound}.
To state the precise statement, we first recall some definitions.

The Poisson bracket invariant of a finite open cover $\mathcal{U}=\{ U_j \}$ of $M$ is
\[
pb(\mathcal{U} ) = \inf_{\{  \chi_j \} }\max_{a_j,b_j\in[-1,1]} ||\{ \sum_j a_j \chi _j , \sum_j b_j \chi_j  \}||,
\]
where $\{ \cdot,\cdot \}$ denotes the Poisson bracket and infimum is taken over all partitions of unity $\{ \chi_j \}$ subordinate to $\mathcal{U}$.
Lower bounds for this invariant is important in the theory of quantum noise in \cite{Pol14}.

We say the degree of $\mathcal{U}$ is $\leq d$ if every subset $\overline{U_j}$ intersects
closures of at most $d$ other subsets from the cover.

The following proposition is the extension.
\begin{prop}
Let $(M,\omega)$ be a closed symplectic manifold satisfying $c_1 = \kappa \omega$ on
$\pi_2(M)$, and $\mathcal{U}= \{ U_j \}$ be a finite open cover with the degree $\leq d$.
Suppose each $U_j$ is symplectomorphic to a strictly convex open subset in $(\R^{2n},\omega_0)$.
\begin{itemize}
\item
If $\kappa\leq 0$, or $\kappa>0$ and $\max_j C_0(U_j) \leq \frac{1}{\kappa}$, then
\[
pb(\mathcal{U})\geq \frac{1}{2(d+1)^2\max_j C_0(U_j)}
\]
\item
If $\kappa>0$ and $\max_j C(U_j) \leq \frac{n}{\kappa}$, then
\[
pb(\mathcal{U})\geq \frac{1}{2(d+1)^2\max_j C(U_j)}
\]
\end{itemize}
\end{prop}
We can prove the above proposition in the same way as Theorem 9 of \cite{Sey14},
using our Proposition \ref{bound} and Remark \ref{variations} instead of Theorem 2 of  \cite{Sey14}.

\section{Another application}\label{another}
Let $\rho : T^k \to \Ham (M,\omega)$ be a Hamiltonian torus action on a closed symplectic manifold and assume $c_1= \kappa \omega$ on $\pi_2(M)$, where $\kappa >0$. Let $\mu : M \to \mathfrak{t}^{\ast}$ be its momentum map. We normalize $\mu$ by
$\int_M \mu \omega^{\wedge n} = 0$.

Define for a loop $g(t) = \phi^H_t : S^1 \to \mathrm{Ham}(M,\omega)$,
\[
I(g) = \mathcal{A}_H([x,u]) + \frac{1}{2\kappa} \CZ _H[x,u],
\]
where $x \in M$ is arbitrary and $H_t$ is a normalized Hamiltonian i.e.\ $\int_M H_t \omega^{\wedge n} = 0$ for every $t\in S^1$. $I(g)$ is independent of $x$ and $u$.
This define a homomorphism $I : \pi_1(\mathrm{Ham}(M,\omega)) \to \R$, which is called the mixed action-Maslov homomorphism in \cite{Pol97} and \cite{EP09}.

In \cite{EP09}, $\mu^{-1}(p_{\star})\subset M$, where $p_{\star} = I\circ \rho_{\ast}\in \Hom (\pi_1(T^k);\R) \cong \mathfrak{t}^{\ast}$, 
is called the special fiber. Entov and Polterovich showed in \cite{EP09} that the special fiber is suerheavy with respect to every non-zero idempotent of $QH_{\ast}(M,\omega)$.
The proof given in \cite{EP09} uses some calculation of the action functional and the Conley-Zehnder index, but sharper estimates can be obtained by our method,
which gives a bound of the Hofer-Zehender capacity of some open subset of $M$. (See Corollary \ref{capacity}.)

\begin{prop}\label{caltorus}
In the above setting,
if $F\in C(M)$ satisfies $F|_{\mu^{-1}(p_{\star})}=0$, then
\[
0\leq c([M];F) \leq \frac{n}{\kappa}.
\]
\end{prop}
\begin{proof}
First we do some preparatory analyses.

Consider a Hamiltonian of the form $H = f\circ \mu : M\to \R$, where $f : \mathfrak{t}^{\ast} \to \R$ is an arbitrary smooth function. Its flow is $\phi_t^H(x) = e^{tdf(\mu(x))}\cdot x$, where we regard the differential of $f$ as $df : \mathfrak{t}^{\ast} \to \mathfrak{t}$.

First we consider the case $f(p)=c + p\cdot \frac{m}{N}$ is linear and $\frac{m}{N} \in \Q^k$.(We regard $\pi_1(T^k)\subset \mathfrak{t}$ as $\Z^k \subset \R^k$.)
In this case, we define $\widetilde{A}_H : \Fix e^{\frac{m}{N}} \to \R$ by
\[
\widetilde{A}_H(x) = \mathcal{A}_H([x,u]) + \frac{1}{2\kappa}(n+n+\frac{1}{N}\CZ_{e^{tm}}[x^{\sharp N},u^{\sharp N}]),
\]
where $x^{\sharp N}(t)=x(Nt)$ and $u^{\sharp N}(z) = u(z^N)$.
Then $\widetilde{A}_H$ is locally constant on $\Fix e^{\frac{m}{N}}$.
We calculate this constant.

Each connected component of $\Fix e^{\frac{m}{N}}$ intersects with
$\Fix (e^{tm})_{t\in S^1}$ since the Hamiltonian $(e^{tm})_{t\in S^1}$-action on $\Fix e^{\frac{m}{N}}$ has fixed points on each component.
On $\Fix (e^{tm})_{t\in S^1}$,
\[
\widetilde{A}_H(x) = H(x) + \frac{1}{2\kappa N}\CZ (e^{tm}_{\ast x})_{t\in[0,1]} + \frac{n}{\kappa}.
\]
By the definition of $p_{\star}$, if $x\in \Fix (e^{tm})_{t\in S^1}$ and $p= \mu (x)$, then
\[
p_{\star}\cdot m=I(e^m) = p\cdot m + \frac{1}{2\kappa}\CZ (e^{tm}_{\ast x})_{t\in[0,1]},
\]
which implies
\[
\widetilde{A}_H(x) =H(x) + \frac{m}{N}\cdot (p_{\star}-p) + \frac{n}{\kappa} = f(p_{\star}) + \frac{n}{\kappa}.
\]
Hence we conclude $\widetilde{A}_H\equiv f(p_{\star}) + \frac{n}{\kappa}$ on
$\Fix e^{\frac{m}{N}}$.

Next we consider an arbitrary $f : \mathfrak{t}^{\ast} \to \R$.
Define $\widehat{A}_H :
\{ x\in \Omega _0 M ; x $ is a contractible periodic orbit of $H \} \to \R$ by
\[
\widehat{A}_H(x) = \mathcal{A}_H([x,u]) + \frac{1}{2\kappa} (n + \max \CZ _H[x,u]).
\]
By Corollary \ref{moment} and the lemma below, for every $[x,u]\in \Crit \mathcal{A}_H$
such that $df(\mu(x)) = \frac{m}{N}\in \Q^k$,
\begin{align*}
\max \CZ _H[x,u] &\leq \max \CZ _{e^{\frac{tm}{N}}}[x,u]\\
&\leq n + \frac{1}{N}\CZ _{e^{tm}}[x^{\sharp N},u^{\sharp N}].
\end{align*}
Hence
$\widehat{A}_H(x) \leq \widetilde{A}_{H_p}(x)$,
where $H_p= f_p\circ \mu$ and $f_p(q)=f(p) + df(p)\cdot (q-p)$ is a tangent of $f$ at $p=\mu (x)$.
Therefore we conclude that
\[
\widehat{A}_H(x)\leq f_p(p_{\star}) + \frac{n}{\kappa}.
\]
If $[x,u]\in \Crit \mathcal{A}_H$ satisfies $df(\mu(x)) = X\in \mathfrak{t}\setminus \Q^k$,
then $x\in \Fix (e^{tX})_{t\in\R}$.
For any $\epsilon >0$ there exists $\frac{m}{N} \in \Q^k\cap T_e(\overline{(e^{tX})_{t\in\R}})$ such that $|X- \frac{m}{N}| \leq \epsilon$,
\begin{align*}
\max \CZ _H[x,u] &\leq \max \CZ _{e^{tX}}[x,u]\\
&\leq n + \frac{1}{N}\CZ _{e^{tm}}[x^{\sharp N},u^{\sharp N}] + 2n\epsilon,
\end{align*}
which implies
\begin{align*}
\widehat{A}_H(x) &\leq H(x) + \frac{n}{\kappa} + \frac{m}{N}\cdot (p_{\star} - p) + \frac{n\epsilon}{\kappa}\\
&= f_p(p_{\star}) + \frac{n}{\kappa} + \frac{n\epsilon}{\kappa}
\end{align*}
by the same argument as above.
Since $\epsilon>0$ is arbitrary, we obtain
\[
\widehat{A}_H(x)\leq f_p(p_{\star}) + \frac{n}{\kappa}.
\]

Let $f : \mathfrak{t}^{\ast} \to \R$ be a smooth convex function which takes its minimum $f=0$
at $p=p_{\star}$.
Then the above argument implies
$
\widehat{A}_H \leq \frac{n}{\kappa}
$
on $\Crit \mathcal{A}_H$.
Hence this implies $c([M];H)\leq \frac{n}{\kappa}$ as in the proof of Proposition \ref{bound}.

If $F\in C(M)$ vanishes on $\mu^{-1}(p_{\star})$,
then for any $\epsilon >0$ there exists a smooth convex function $f : \mathfrak{t}^{\ast} \to\R$
such that $f$ takes its minimum $f=0$ at $p=p_{\star}$ and $F \leq \epsilon + f\circ \mu$,
which implies $c([M];F)\leq \frac{n}{\kappa}$.
In particular, $M\setminus \mu^{-1}(p_{\star})$ is $[M]$-superheavy, which implies $c([M];F)\geq 0$.
\end{proof}
\begin{lem}
For any $y\in \R$ and integers $m\in \Z$, $N>0$,
\[
\max \CZ (e^{-2\pi y \sqrt{-1}t})_{t\in[0,1]}\leq \frac{1}{N} \CZ (e^{-2\pi m\sqrt{-1}t})_{t\in[0,1]} + 1 + 2\biggl|y - \frac{m}{N}\biggr|
\]
In particular,
\[
\max \CZ (e^{-2\pi\frac{m}{N}\sqrt{-1}t})_{t\in[0,1]}\leq \frac{1}{N} \CZ (e^{-2\pi m\sqrt{-1}t})_{t\in[0,1]} + 1
\]
\end{lem}
\begin{proof}
\begin{align*}
\max \CZ (e^{-2\pi y \sqrt{-1}t})_{t\in[0,1]}&\leq -1-2[ y ]^<\\
&\leq \frac{1}{N}(-2m) + 1 + 2\biggl|y- \frac{m}{N}\biggr|\\
&= \frac{1}{N} \CZ (e^{-2\pi m\sqrt{-1}t})_{t\in[0,1]} + 1 + 2\biggl|y- \frac{m}{N}\biggr|
\end{align*}
\end{proof}

Recall the Hofer-Zehnder capacity $c_{HZ}(M,\omega)$ of a symplectic manifold $(M,\omega)$ is defined by
\begin{align*}
c_{HZ}(M,\omega) = \sup \{ \max H ; &H\in C_0^{\infty}(\mathrm{Int}M),\ \min H = 0, \\
&H \text{ has no periodic orbit with period }0< T \leq 1\}.
\end{align*}

\begin{cor}\label{capacity}
Let $(N,\omega')$ be an arbitrary closed symplectic manifold.
Then under the assumptions of Proposition \ref{caltorus},
\[
c_{HZ}(M\setminus\mu^{-1}(p_{\star}))\leq c_{HZ}((M\setminus\mu^{-1}(p_{\star}))\times N) \leq \frac{n}{\kappa}.
\]
\end{cor}
\begin{proof}
This is a consequence of Proposition \ref{product} and the fact that for any open subset $U\subset M$ of a closed symplectic manifold $(M,\omega)$,
\[
c_{HZ}(U)\leq \sup_{\mathrm{supp}F\subset U}c([M];F).
\]
(See \cite{Ush10}.)

\end{proof}

\begin{eg}\label{proj}
Let $(\C P^n,\tau _0)$ be the complex projective space with the Fubini-Study form and consider the Hamiltonian $T^n=T^{(n+1)}/S^1$-aciton on $(\C P^n,\tau _0)$ given by
$t\cdot [z_0:z_1:\dots:z_n] = [t_0z_0:t_1z_1:\dots:t_nz_n]$.
Then the special fiber is the Clifford torus $C=\{ |z_0|=|z_1|=\dots=|z_n| \}$.
Hence the above proposition implies
\[
D(\C P^n\setminus C) = c_{HZ}(\C P^n\setminus C) = c_{HZ}(B),
\]
where $D(U,\omega)=\sup \{\pi r^2; (B^{2n}(r),\omega_0)$ can be symplectically embedded in $(U,\omega)$\} stands for the Gromov width and
$B=\{ \frac{|z_n|^2}{\sum |z_j|^2} > \frac{1}{n+1}\}$ is a ball in $\C P^n\setminus C$.
\end{eg}

\section*{Acknowledgments}
We are grateful to S. Seyfaddini and L. Polterovich for telling us the 
relation to their results.
We thank K. Ono and K. Irie for useful comments and discussions.
We also thank referees for their useful suggestions.
Their comments improved this paper, especially the proof of Lemma \ref{delta}
was simplified by their suggestion.

Research Institute for Mathematical Sciences, Kyoto University, Kyoto, Japan
\\
suguru@kurims.kyoto-u.ac.jp

\end{document}